\documentclass{article}
\usepackage{plain,amssymb,amsfonts,amsthm,amsmath}
\RequirePackage{hyperref}

\theoremstyle{plain}
\newtheorem{theorem}{Theorem}[section]

\newtheorem{lemma}[theorem]{Lemma}
\newtheorem{cor}[theorem]{Corollary}
\newtheorem{example}{Example}[section]
\newtheorem{assumption}{Assumption}

\long\def\comment#1{{}}

\def\Grp#1{\left(#1\right)}
\def\Cbr#1{\left\{#1\right\}}
\def\Sbr#1{\left[#1\right]}
\def\Abs#1{\left|#1\right|}
\def\Norm#1{\left\|#1\right\|}
\def\normx#1{\|#1\|}
\def\lipnorm#1{\Norm{#1}_{\rm Lip}}

\def\nth#1{\frac{1}{#1}}
\def\cf#1{\mathbf{1}\Cbr{#1}}
\def\tp{\sp{\top}}
\def\eno#1#2{{#1}_1, \ldots, {#1}_{#2}}
\def\gv{\,|\,}

\def\sm #1#2{\sum_{#1\le #2}}
\def\mx #1#2{\max_{#1\le #2}}
\def\Sp#1{\sp{(#1)}}

\def\mean{\text{\sf E}}
\def\var{\text{\sf Var}}
\def\prob{\text{\sf Pr}}
\def\Reals{\mathbb{R}}

\def\dev#1{\Sbr{\hspace{-.65ex}\Sbr{#1}\hspace{-.65ex}}}
\def\sdev#1{\Sbr{\hspace{-.35ex}\Sbr{#1}\hspace{-.35ex}}}
\def\est#1{\widehat{#1}}
\def\sppt{{\rm spt}}
\def\dd{\mathrm{d}}
\def\risk{L}

\def\cT{\mathcal{T}}

\def\cY{\mathcal{Y}}
\def\cZ{\mathcal{Z}}

\def\rx{\varepsilon}
\def\gx{\omega}

\def\phidef{\min\Grp{\frac{2F_m}{m!}, \frac{F_{m+1} R}{(m+1)!}}}
\def\psidef{
  \begin{cases}
    F_{m+1}/m! & m\not=1 \\
    F_{m+1}/2 & m=1.
  \end{cases}
}

\begin{document}

\begin{center}
  \Large{\bf
    A local stochastic Lipschitz condition with application to Lasso
    for high dimensional generalized linear models
  }
  \\[1em]
  \normalsize
  Short title: LSL and Lasso \\[1em]
  Zhiyi Chi \\
  Department of Statistics\\
  215 Glenbrook Road, U-4120 \\
  Storrs, CT 06269, USA \\[1em]
  \today
\end{center}

\begin{abstract}
  For regularized estimation, the upper tail behavior of the random
  Lipschitz coefficient associated with empirical loss functions
  is known to play an important role in the error bound of Lasso
  for high dimensional generalized linear models.  The upper tail
  behavior is known for linear models but much less so for nonlinear
  models.  We establish exponential type inequalities for the upper
  tail of the coefficient and illustrate an application of the results
  to Lasso likelihood estimation for high dimensional generalized
  linear models.

  \medbreak\noindent
  
  \medbreak\noindent
  
  \medbreak\noindent
  \textit{Acknowledgement.\/} Research partially supported by NSF
  grant DMS-07-06048.
\end{abstract}

\section{Introduction}  \label{sec:intro}
Let $(Y_1, Z_1)$, \ldots, $(Y_N, Z_N)$ be independent random variables
taking values in a product measurable space $\cY\times \cZ$, with
$Y_i$ being regarded as response variables and $Z_i$ as covariates.
In order to cover both random designs and fixed designs, $(Y_i, Z_i)$
are not necessarily identically distributed.  A large class of Lasso
type estimators for high dimensional generalized linear models can be
formulated as
\begin{align} \label{eq:lasso-0}
  \est\theta = \mathop{\arg\min}_{v\in D_0} \Cbr{
    \sm i N [\gamma_i(h(Z_i)\tp v, Y_i) + b(v)] + \sm j p \lambda_j
    |v_j|
  },
\end{align}
where $D_0\not=\emptyset$ is a domain in $\Reals^p$, $\gamma_i(t,y)$
are a given set of real valued functions on $\Reals\times\cY$, 
oftentimes identical to each other, $h=(\eno h p): \cZ\to\Reals^p$ and
$b: D_0\to \Reals$ are given functions, and $\eno\lambda
p >0$ are coefficients of the weighted $\ell_1$ penalty on $v$.  In
this article, we only consider nonadaptive Lasso, in which $\eno
\lambda N$ are fixed beforehand.

Under the setting of \eqref{eq:lasso-0}, for each $v\in D_0$, we have
$N$ loss functions, each defined as $(y,z)\to \gamma_i(h(z)\tp v, y) +
b(v)$.  The corresponding empirical losses are $\gamma_i(h(Z_i)\tp v,
Y_i) + b(v)$, and the corresponding 
expected total loss is
\begin{align} \label{eq:loss-lasso}
  \risk(v) = \sm i N \mean[\gamma_i(h(Z_i)\tp v, Y_i) + b(v)], \quad
  v\in D_0.
\end{align}

As the title suggests, the main interest of the article is the so
called ``local stochastic Lipschitz'' (LSL) condition.   By LSL we
mean the following.   For the time being, denote by
$$
\tilde\risk(v) = \sm i N [\gamma_i(h(Z_i)\tp v, Y_i) + b(v)] - L(v)
$$
the fluctuation of the empirical total loss from its expectation at
parameter value $v$.  Let
$\theta\in\Reals^p$ be fixed.  Under smooth conditions for $\gamma_i$,
it is easy to see $\tilde\risk(v)$ is differentiable with probability
(w.p.) 1, which in general leads to Lipschitz continuity 
of $\tilde\risk(v)$ provided $D_0$ is compact.  The LSL condition, on
the other hand, refers to a bound on the upper tail probability of the
random variable
\begin{align} \label{eq:local-lip}
  \sup_{v\in D_0,\, v\not=\theta}
  \frac{|\tilde\risk(v) - \tilde\risk(\theta)|}
  {\sm j p \lambda_j |v_i-\theta_j|}.
\end{align}
Note that the LSL condition is with respect to a weighted $\ell_1$
norm of $\Reals^p$.  The condition is called ``local'' because
$\theta$ is fixed, even though its value is typically unknown.

Although it might not be apparent at this point, the LSL condition is
closely related to the issue of estimation error for Lasso.  For
linear regression with square loss function $(y - h(z)\tp v)^2$,
this relationship is well known and has been regularly employed to 
obtain estimation error bounds \cite{candes:plan:09,
  bunea:10:as, bunea:07:as, bickel:09:as}.  Indeed, in this case, due
to linearity, the LSL condition is rather easy to establish.  However,
for other loss functions, the LSL condition is much less clear and, to
my best knowledge, has not been fully explored.  An alternative to
the LSL condition is a convexity assumption, in which $\gamma_i(t,y)$
is convex in $t$ and $b(v)$ is convex in $v$.  The convexity
assumption allows a linear interpolation technique to be employed to
yield upper bounds for estimation error \cite{vandegeer:08}.  While
the convexity assumption allows for nondifferentiable $\gamma_i$, it
is not clear how the technique can be extended to nonconvex loss
functions.

We shall establish the LSL condition for general loss functions.  For
differentiability, we only require that 
$\gamma_i(t,y)$ be first order differentiable in $t$ with the partial
derivative being Lipschitz.  After getting various results
on the LSL condition, we will then illustrate an application of the
LSL condition to Lasso type nonlinear regression, by finding an upper
bound for the $\ell_2$ norm of estimation error.

Previously, in \cite{chi:10}, the LSL condition was studied for loss
functions of the form $(y-g_i(h(z)\tp v))^2$, $i\le N$, where
$g_i:\Reals\to \Reals$ are nonlinear.  The condition was established
under the assumptions that $g_i$ are twice continuously differentiable
and
\begin{align} \label{eq:glm-add}
  Y_i =  g_i(h(Z_i)\tp\theta) + \rx_i,
\end{align}
where $\rx_i$ are uniformly bounded zero mean noise.  In this article,
we extend the result on two aspects.  First, the LSL condition is
established for general $\gamma_i(t,y)$, while still under the
assumption of uniform boundedness.  Second, it is established for
\eqref{eq:glm-add} when $\rx_i$ are Gaussian.  Whereas the bounds for
general $\gamma_i(t,y)$ is of Bernstein type, the bounds for the
Gaussian case is of Hoeffding type.  In \cite{chi:10}, a truncation
argument was suggested for the Gaussian case.  However, the LSL
condition obtained in this way is not as tight as the one to be
obtained here.  The tools used to get the results on the LSL condition
are various measure concentration and comparison inequalities in
Probability \cite{ledoux:91, ledoux:01, klein:05:ap}.

Section \ref{sec:lip} presents several results on the LSL condition.
The discussion in the section is actually more general.  It provides
upper bounds on the tail probability of the remainder of the Taylor
expansion of $\tilde\risk(v)$.  The LSL condition is a simple
consequence of these bounds.

In Section \ref{sec:app-lasso}, we consider an application of the LSL
condition to Lasso.  Besides the LSL condition, Lasso involves another
issue, that is, the amount of separation of $v$ and $\theta$ based on
the difference between $\gamma_i(h(Z_i)\tp v, Y_i)$ and
$\gamma_i(h(Z_i)\tp \theta, Y_i)$.  This issue is of different nature
from the LSL condition, and its resolution in general requires further
conditions on the matrix $[h_j(Z_i)]_{i\le N, j\le p}$.  The issue has
been studied in quite a few works \cite{zhao:yu:06, bunea:07:as,
  candes:tao:07, bickel:09:as, zhang:09}.  For transparency, we will
use a restricted eigenvalue condition in \cite{bickel:09:as} for our
purpose.  We will consider an example of Lasso type MLE for high
dimensional generalized linear model and apply the LSL condition
to bound the $\ell_2$ norm of the estimation error.  Unfortunately,
the method of the example gives no clue on model selection or more
elaborate bounds similar to those obtained for linear models under
square loss \cite{zhang:09,bickel:09:as, candes:plan:09}.   All the
proofs are presented in Section \ref{sec:proof}.

\subsection{Notation}
For $q\in [1,\infty)$, denote by $\normx{a}_q$ the $\ell_q$ norm of
$a\in \Reals^d$.  For two vectors $a=(\eno a m)\tp$ and $b=(\eno b
n)\tp$, recall that their tensor product is
$$
a\otimes b = (a_1 b\tp, \ldots, a_m b\tp)\tp =
(a_1 b_1, \ldots, a_1 b_n, \ldots, a_n b_1, \ldots, a_m b_n)\tp \in
\Reals^{mn}. 
$$
Denote by $v^{\otimes k}$ the tensor product of $k$ copies of $v$.

If $f$ is a function on a domain $\Omega\subset\Reals^d$, then
it is Lipschitz (under the Euclidean norm) if
$$
\lipnorm{f}:=
\sup_{x\not=y\in\Omega}\frac{|f(x)-f(y)|}{\normx{x-y}_2} <\infty.
$$

Finally, for any random vector $X$, denote its deviation from mean by
$$
\sdev{X} = X - \mean X.
$$
By linearity of expectation, $\sdev{X+Y} = \sdev{X} + \sdev{Y}$.  By
this notation,
$$
\tilde\risk(v) = \sm i N\dev{\gamma_i(h(Z_i)\tp v, Y_i)}.
$$
The right hand side is independent of $b(v)$ and at the same time
better reveals the other quantities involved.  We will discard the
notation $\tilde\risk$ in favor of $\sdev{\cdot}$ for the rest of the
article.

\subsection{Notes}
The methods in Section \ref{sec:lip} can be used
with little change to deal with the following additive mixture of loss
functions,
$$
\sm i N \sm k q [\gamma_{ik}(h_k(Z_i)\tp v, Y_i) + b_k(v)]
$$
where for each $k\le q$ and $i\le N$, $h_k = (h_{k1}, \ldots, h_{kp})$
is a function from $\cZ$ to $\Reals^p$, and $\gamma_{ik}$ is a loss
function.  For example
$$
\sm i N \gamma_i (h(Z_i)\tp v, Y_i) 
+\sm i N \tilde\gamma_i (\tilde h(\tilde Z_i)\tp u, Y_i)
$$
is a special case of additive mixture, where $Z_i$ and $\tilde Z_i$
are covariates that may be identical or have completely different sets of
coordinates.  Due to identifiability issue in the context of parameter
estimation, such mixtures will not further considered in the article.

\section{Local stochastic Lipschitz condition} \label{sec:lip}
In this section, we present exponential bounds on the tail probability
of the random local Lipschitz coefficient \eqref{eq:local-lip}.  As
noted earlier, these bounds are consequences of more general results
on the tail probability of remainders of Taylor expansion of random
functions.  Therefore, most of the discussion below will be on the
latter and the results on the LSL condition will be given as
corollaries.

\subsection{General loss function}
Suppose $\eno\gamma N$ satisfy the following regularity condition.
\begin{assumption}[Regularity] \label{a:regular}
  There are $m\in \{0, 1, 2, \ldots\}$ and $-\infty \le a_i < b_i \le
  \infty$, $i\le N$, such that w.p.~1, each
  $\gamma_i(t, Y_i)$ as a function of $t$ is $m$ times differentiable
  on $(a_i, b_i)$ with the $m$-th derivative being bounded and
  Lipschitz.  Let $F_m$, $F_{m+1}$ be constants such that w.p.~1, 
  \begin{gather*}
    \begin{cases}
      \displaystyle
      \Abs{
        \frac{\partial^m\gamma_i(t, Y_i)}{\partial t^m}
      } \le F_m, \\[3ex]
      \displaystyle
      \Abs{
        \frac{\partial^m \gamma_i(t, Y_i)}{\partial t^m}
        -\frac{\partial^m \gamma_i(t', Y_i)}{\partial t^m}
      } \le F_{m+1}|t-t'|,
    \end{cases}
    \quad
    \forall\, t, t'\in (a_i,b_i), \ i\le N.
  \end{gather*}
\end{assumption}

Suppose $h$ satisfies the following condition.
\begin{assumption}[Boundedness] \label{a:bounded}
  There are constants $\eno d p\in (0,\infty)$, such that
  $$
  \prob\Cbr{\mx i N |h_j(Z_i)|\le d_j,\ \forall\, j\le p}=1.
  $$
\end{assumption}

Next, let $D_0\not=\emptyset$ be a domain in $\Reals^p$.
\begin{assumption}[Parameter Domain] \label{a:domain}
  For $(a_i,b_i)$ as in Assumption \ref{a:regular} and $h$ as in
  Assumption \ref{a:bounded}, 
  $$
  \prob\Cbr{h(Z_i)\tp v\in (a_i,b_i), \,\forall\,v\in D_0,\ i\le
    N}=1.
  $$
\end{assumption}

From Assumption \ref{a:regular} and dominated convergence,
differentiation and expectation can be exchanged for $\gamma_i(t,
Y_i)$, i.e.,
$$
\mean \Sbr{\frac{\partial^k \gamma_i(t, Y_i)}{\partial t^k}}
= \frac{\partial^k \mean[\gamma_i(t, Y_i)]}{\partial t^k}, \quad
t\in (a_i,b_i),\ i\le N, \ k\le m.
$$
By Assumption \ref{a:bounded}, $|h_j(Z_i)/d_j|\le 1$ w.p.~1.
Therefore, $d_j$ can 
be thought of as the ``scales'' of the functions $h_j$.

\begin{theorem} \label{thm:lipschitz}
  Under Assumptions \ref{a:regular} -- \ref{a:domain}, fix an
  arbitrary $\theta\in D_0$.  Then for $v\in D_0$,
  \begin{align}
    &
    \sm i N \dev{\gamma_i(h(Z_i)\tp v, Y_i)}
    \nonumber \\
    &
    =
    \sm k m \nth{k!}
    \sm i N \dev{
      \frac{\partial^k \gamma_i(h(Z_i)\tp \theta, Y_i)}{\partial t^k}
      [h(Z_i)\tp(v-\theta)]^k
    }
    + \xi(v) \Grp{\sm j p d_j|v_j-\theta_j|}^m
    \label{eq:taylor-1} \\
    &
    =
    \sm k m \nth{k!}
    \sm i N \dev{
      \frac{\partial^k \gamma_i(h(Z_i)\tp \theta, Y_i)}{\partial t^k}
      h(Z_i)^{\otimes k}
    }\tp (v-\theta)^{\otimes k}
    + \xi(v) \Grp{\sm j p d_j|v_j-\theta_j|}^m,
    \label{eq:taylor-2}
  \end{align}
  where $\{\xi(v), v\in D_0\}$ is a process that has the following
  upper tail property
  \begin{align*}
    \prob\Cbr{\sup_{v\in D_0} |\xi(v)| > A \sqrt{2\ln (2p)} + 
      B\sqrt{2\ln(p^m/q)} + C\ln(p^m/q)
    } \le q, \quad\forall\, q\in(0,1)
  \end{align*}
  with $A$, $B$, and $C$ being set as follows.  First, let
  \begin{align} \label{eq:R}
    R = \sup_{u,v\in D_0} \sm j p d_j |u_j -v_j|, \quad
    \phi = \phidef, \quad
    \psi = \psidef.
  \end{align}
  Then
  \begin{align*}
    A
    = 8\psi R
    \mean\sqrt{\mx j p \sm i N [h_j(Z_i)/d_j]^2}, \quad
    B
    = \phi\sqrt{\mean\mx j p \sm i N [h_j(Z_i)/d_j]^{2m}}, \quad
    C=8\phi,
  \end{align*}
  where in the definition of $B$ the convention $x^0\equiv 1$ is
  used for $m=0$.
\end{theorem}

Note that if $F_{m+1}>0$, then the above result is meaningful only
when $R<\infty$, that is, $D_0$ is bounded.  On the other hand, if
w.p.~1, for $i\le N$, $\gamma_i(t,Y_i)$ is a linear function of $t$,
then one can set $F_{m+1}=0$.  By Theorem \ref{thm:lipschitz},
this yields $A=B=C=0$, which implies $\xi(v)\equiv 0$.  Of
course, the last fact is easy to be seen by the linearity of
$\gamma_i(t,Y_i)$.

Of particular interest is the case where $m=1$.  From Theorem
\ref{thm:lipschitz}, the following result obtains.
\begin{cor} \label{cor:m=1}
  Under Assumptions \ref{a:regular} -- \ref{a:domain} with $m=1$, fix
  an arbitrary $\theta\in D_0$.  Then for $v\in D_0$,
  \begin{align} \label{eq:taylor-m=1}
    \sm i N \dev{\gamma_i(h(Z_i)\tp v, Y_i)}
    = \sm i N \dev{\gamma_i(h(Z_i)\tp \theta, Y_i)}
    + [\xi_1+\xi(v)]\sm j p d_j|v_j - \theta_j|
  \end{align}
  where $\xi(v)$ is as in Theorem \ref{thm:lipschitz} and $\xi_1$ is
  a random variable with the following upper tail property
  \begin{align*}
    \prob\Cbr{|\xi_1| >  F_1\sqrt{2N\ln(2p/q)}}\le q, \quad
    \forall\, q\in(0,1).
  \end{align*}
\end{cor}

Since
$$
|\xi_1|+\sup_{v\in D_0} |\xi(v)|
\ge 
\sup_{v\in D_0,\, v\not=\theta}
\nth{\sm j p \lambda_j |v_i-\theta_j|}
\Abs{
  \sm i N \dev{
    \gamma_i(h(Z_i)\tp v, Y_i)- \gamma_i(h(Z_i)\tp \theta, Y_i)
  }
},
$$
from the result, we then get a desired form of the LSL condition.  For
any $q$, $q'\in (0,1)$ not necessarily equal, one can find $M(q,q')$,
such that w.p.~at least $1-q-q'$, the random local Lipschitz
coefficient on the right hand side is no greater than $M(q,q')$.
Moreover, one can set
$$
M(q,q')=A\sqrt{2\ln(2p)}+B\sqrt{2\ln(p/q)} + C\ln(p/q) + F_1\sqrt{2N
  \ln(2p/q')},
$$
with $A$, $B$ and $C$ given as in Theorem \ref{thm:lipschitz} with
$m=1$.
\subsection{Gaussian case}
Suppose $\eno Z N$ are fixed and
$$
Y_i = \mu_i - \gx_i
$$
where $\mu_i$ are some unknown constants, and $\eno\gx N$ are
independent square-integrable random variables with mean 0.  Let
$\eno f N: \Reals\to\Reals$ be a set of transforms specified
beforehand, and $h=(\eno h p): \cZ\to \Reals^p$ a measurable
function.  Suppose the goal is to use $f_i(h(Z_i)\tp v)$ to
approximate $\mu_i$ under the square loss functions 
\begin{align} \label{eq:loss-sq}
  \gamma_i(t, Y_i) = (Y_i - f_i(t))^2/2.
\end{align}

For any $v$, provided that $h(Z_i)\tp v$ is in the domain of $f_i$ for
all $i\le N$,
\begin{align*}
  \dev{\gamma_i(h(Z_i)\tp v, Y_i)}
  &
  = \nth 2(\mu_i - \gx_i - f_i(h(Z_i)\tp v))^2 - \nth 2\mean[
  (\mu_i - \gx_i - f_i(h(Z_i)\tp v))^2] \\
  &
  = \gx_i [f_i(h(Z_i)\tp v)-\mu_i] + \nth 2[\gx_i^2 - \var(\gx_i)].
\end{align*}
Thus, for any $\theta$, provided that $h(Z_i)\tp \theta$ is in the
domain of $f_i$ for all $i\le N$ as well
$$
\sm i N\dev{\gamma_i(h(Z_i)\tp v, Y_i)}
-\sm i N\dev{\gamma_i(h(Z_i)\tp \theta, Y_i)}
= \sm i N \gx_i \Sbr{f_i(h(Z_i)\tp v) - f_i(h(Z_i)\tp\theta)}.
$$
As a result, we will focus on the expansion of the random function
$$
v\to \sm i N \gx_i f_i(h(Z_i)\tp v)
$$
around any fixed $\theta\in D_0$.

\begin{assumption}[Regularity] \label{a:regular-g}
  There are $m\in \{0, 1, 2, \ldots\}$ and $-\infty \le
  a_i<b_i\le\infty$, $i\le N$, such that each $f_i$ is $m$ times
  differentiable on $(a_i, b_i)$ with the $m$-th derivative being
  bounded and Lipschitz.   Let
  \begin{gather*}
    F_m = \mx i N \sup_{t\in (a_i, b_i)} |f_i\Sp m(t)|,
    \quad
    F_{m+1} = \mx i N \lipnorm{f_i\Sp m}.
  \end{gather*}
\end{assumption}

Since $Z_i$ are fixed, Assumption \ref{a:bounded} is no longer needed.
Instead, simply define
$$
d_j = \mx i N |h_j(Z_i)|.
$$
Also, modify Assumption \ref{a:domain} as follows.
\begin{assumption}[Parameter Domain] \label{a:domain-g}
  The domain $D_0\not=\emptyset$ of candidate parameter values
  satisfies $h(Z_i)\tp v\in (a_i, b_i)$, $\forall v\in D_0$, $i\le
  N$.
\end{assumption}

In \cite{chi:10}, the case where $\gx_i$ are uniformly bounded is
considered.  Here we shall deal with the following situation.
\begin{assumption}[Gaussian] \label{a:noise-g}
  $\eno\gx N$ are independent Gaussian variables with $\var(\gx_i)\le
  \sigma_0^2$, $i\le N$, where $\sigma_0\in (0,\infty)$ is a
  constant.
\end{assumption}

\begin{theorem} \label{thm:lipschitz-g}
  Let the loss functions $\eno\gamma N$ be as in \eqref{eq:loss-sq}.
  Under Assumptions \ref{a:regular-g} -- \ref{a:noise-g}, fix an
  arbitrary $\theta\in D_0$.  Then for $v\in D_0$,
  \begin{align}
    &
    \sm i N \gx_i f_i(h(Z_i)\tp v)
    \nonumber \\
    &
    =
    \sm k m \nth{k!}
    \Grp{\sm i N 
      \gx_i f_i\Sp k(h(Z_i)\tp\theta)[h(Z_i)\tp(v-\theta)]^k
    }
    + \xi(v) \Grp{\sm j p d_j|v_j-\theta_j|}^m
    \label{eq:taylor-1g} \\
    &
    =
    \sm k m \nth{k!}
    \Grp{\sm i N 
      \gx_i f_i\Sp k(h(Z_i)\tp\theta) h(Z_i)^{\otimes k}
    }\tp (v-\theta)^{\otimes k}
    + \xi(v) \Grp{\sm j p d_j|v_j-\theta_j|}^m,
    \label{eq:taylor-2g}
  \end{align}
  where $\{\xi(v), v\in D_0\}$ is a process that has the following
  upper tail property
  $$
  \prob\Cbr{\sup_{v\in D_0} |\xi(v)|>\sigma_0(A\sqrt{\ln (2p)} + 
    B\sqrt{2\ln(p^m/q)})
  } \le q, \quad\forall\, q\in(0,1)
  $$
  with $A$ and $B$ being set as follows.  First, set $R$, $\phi$ and
  $\psi$ as in \eqref{eq:R}.
  Then
  \begin{align*}
    A= 8\psi R \sqrt{\mx j p \sm i N [h_j(Z_i)/d_j]^2}, \quad
    B= \phi\sqrt{\mx j p \sm i N [h_j(Z_i)/d_j]^{2m}},
  \end{align*}
  where in the definition of $B$ the convention $x^0\equiv 1$ is
  used for $m=0$.
\end{theorem}

Comparing to Theorem \ref{thm:lipschitz}, the above upper tail bound
does not have a term of the form $C \ln(p^m/q)$.  This is because in
the Gaussian case, we can get a Hoeffding type inequality for the
upper tail instead of a Bernstein type inequality.

From Theorem \ref{thm:lipschitz-g}, the following result for the
case $m=1$ obtains.  Note that the result is not entirely the same as
Corollary \ref{cor:m=1}.
\begin{cor} \label{cor:m=1g}
  Under Assumptions \ref{a:regular-g} -- \ref{a:noise-g} with $m=1$,
  fix an arbitrary $\theta\in D_0$.  Define positive constants $\eno w
  p$ as
  \begin{align} \label{eq:m=1gw}
    w_j^2 = \sigma_0^{-2}\sm i N \var(\omega_i) h_j(Z_i)^2.
  \end{align}
  Then for $v\in D_0$,
  \begin{align} \label{eq:taylor-m=1g}
    \sm i N \gx_i f_i(h(Z_i)\tp v)
    =
    \sm i N \gx_i f_i(h(Z_i)\tp \theta) + 
    \sigma_0 F_1\xi_1 \sm j p w_j |v_j - \theta_j| +
    \xi(v) \sm j p d_j |v_j - \theta_j|
  \end{align}
  where $\{\xi(v): v\in D_0\}$ is as in Theorem \ref{thm:lipschitz-g}
  and $\xi_1$ is a random variable with the following upper tail
  property
  $$
  \prob\Cbr{|\xi_1|>\sqrt{2\ln(p/q)}
  } \le q, \quad\forall\, q\in(0,1).
  $$
\end{cor}

Similar to Corollary \ref{cor:m=1}, the above result can be used to
get the LSL condition.  For example, for any $q$, $q'\in (0,1)$, one
can set 
$$
M(q,q') =\sigma_0\Sbr{
  A\sqrt{\ln (2p)} + B\sqrt{2\ln(p/q)}+F_1 \sqrt{2\ln(p/q')}
},
$$
with $A$ and $B$ given as in Theorem \ref{thm:lipschitz-g} with $m=1$,
such that w.p.~at least $1-q-q'$, the following random local Lipschitz
coefficient
$$
\sup_{v\in D_0,\, v\not=\theta}
\nth{\sm j p \lambda_j |v_i-\theta_j|}
\Abs{
  \sm i N \gx_i[f_i(h(Z_i)\tp v)- f_i(h(Z_i)\tp\theta)]
}
$$
is no greater than $M(q,q')$, where $\lambda_j = \max(w_j, d_j)$.

\section{An application to high dimensional Lasso}
\label{sec:app-lasso}
Under Assumptions \ref{a:regular} -- \ref{a:domain}, we consider the
case where $\eno Z N$ are fixed.  For simplicity, assume
$d_1=\ldots = d_N = d$ in Assumption \ref{a:bounded}.  Consider the
following Lasso functional
\begin{align} \label{eq:lasso-def}
  \est\theta = \mathop{\arg\min}_{v\in D_0} \Cbr{
    \sm i N \gamma_i(h(Z_i)\tp v, Y_i) + \lambda d \normx{v}_1
  },
\end{align}
where $\lambda>0$ is the tuning parameter.  Suppose $D_0$ is compact
so that the minimum is always obtained.  The goal is to have
$\est\theta$ approximate to $\theta$, where
$$
\theta = \mathop{\arg\min}_{v\in D_0} 
  \sm i N \mean[\gamma_i(h(Z_i)\tp v, Y_i)].
$$

We next consider applying Corollary \ref{cor:m=1} to bound
$\normx{\est\theta-\theta}_2$.  Denote $X_i = h(Z_i)$ and $X$ the
$N\times p$ matrix with $X_i\tp$ as the $i$-th row vector.  The
total expected loss function now can be written as
$$
\risk(v) = \sm i N \mean[\gamma_i(X_i\tp v, Y_i)], \quad v\in D_0.
$$
Denote by $\sppt(v)=\{j\le p: v_j\not=0\}$ and by $\normx{v}_0$ the
cardinality of the set.  In general, in order to bound
$\normx{\est\theta-\theta}_2$, some conditions on $X$ are needed in
order to get a bound in terms of the $\ell_2$ norm of $v-\theta$
(cf. \cite{zhang:09,candes:plan:09,bickel:09:as}).  For transparency,
we use a ``restricted eigenvalue'' condition formulated in
\cite{bickel:09:as}, which says that for some $1\le s\le p$ and $c>0$,
\begin{align*}
  \kappa(s, K):= \min\Cbr{
    \frac{\normx{Xv}_2}{\sqrt{N} \normx{v_J}_2}: 
    1\le |J|\le s, \ v\not=0, \normx{v_{J^c}}_1 \le K \normx{v_J}_1
  } >0.
\end{align*}

To see where the LSL condition is to be used, we first summarize
an argument that has been more or less used for special cases of Lasso
(cf.\ \cite{vandegeer:08, bickel:09:as}).  Note that the argument does
not lead to model selection or more elaborate bounds that have been
obtained especially for linear models under square loss
\cite{candes:tao:07,zhang:09,bickel:09:as, candes:plan:09}. 

\begin{theorem} \label{thm:lasso}
  Suppose the following conditions are satisfied.
  \begin{itemize}
  \item[1)] For some $K>1$,
    \begin{align}  \label{eq:RE}
      \kappa:=\kappa(2\normx{\theta}_0, K)>0
    \end{align}
  \item[2)] For some $C_\gamma>0$,
    \begin{align} \label{eq:diff-R-quadratic}
      \risk(v) - \risk(\theta)
      \ge C_\gamma \normx{X(v-\theta)}_2^2, \quad
      \forall v\in D_0.
    \end{align}
  \item[3)] Given $q\in (0,1)$, suppose there is $M_q>0$, such that
    w.p.\ at least $1-q$,
    \begin{align} \label{eq:lasso-d}
      \Abs{
        \sm i N\dev{\gamma_i(X_i\tp \est\theta, Y_i)-
          \gamma_i(X_i\tp \theta, Y_i)
        }
      }
      \le M_q d \normx{\est\theta - \theta}_1.
    \end{align}
  \end{itemize}
  Then, by setting
  \begin{align} \label{eq:lasso-par}
    \lambda = \frac{(K+1)M_q d}{K-1}
  \end{align}
  in the Lasso functional \eqref{eq:lasso-def}, on the event that
  \eqref{eq:lasso-d} holds,
  \begin{align}
    \normx{\est\theta-\theta}_2 \le
    \frac{M_q\sqrt{\normx{\theta}_0}}{N}
    \times \frac{2\sqrt{2+K^2} K d}{C_\gamma \kappa^2 (K-1)} 
    \label{eq:L2-total}
  \end{align}
\end{theorem}

Theorem \ref{thm:lasso} has three conditions.  The first one is the
aforementioned restricted eigenvalue condition.  In some cases, 
the second condition is easy to establish.  The third
condition is the LSL condition.  By Corollaries
\ref{cor:m=1} and \ref{cor:m=1g}, $M_q$ can be set reasonably small,
ideally of order $\sqrt{N}$ or even smaller.

\def\cF{\mathcal{F}}
\begin{example} \rm
  Let $\cY$ be a Euclidean space and $\cF=\{f(y\gv t):\, y\in\cY,\,
  t\in [a,b]\}$ a family of densities on $\cY$, where $-\infty < a <
  b<\infty$.  Suppose given $Z_i$, the density of $Y_i$ is 
  $$
  f(y\gv X_i\tp\theta)
  $$
  where $\theta$ is the parameter and $X_i$ again is $h(Z_i)$.
  Suppose it is known that $\theta\in D_0$, where $D_0\subset
  \Reals^p$ is an open bounded region such that for $v\in D_0$,
  $X_i\tp v\in [a,b]$ for each $i\le N$.  Then any solution
  $\est\theta$ to \eqref{eq:lasso-def} with
  \begin{align} \label{eq:mle-loss}
    \gamma_i(t,y) = -\ln f(y\gv t) := \ell(t,y), \quad i\le N
  \end{align}
  is an $\ell_1$ regularized MLE of $\theta$.  Suppose $X$ satisfies
  \eqref{eq:RE}.  We next find some conditions in order for
  \eqref{eq:diff-R-quadratic} to hold.  Let 
  $I(t)$ denote the Fisher information of $\cF$ at $t$ and 
  $$
  D(t,s) = \int f(y\gv t)\ln\frac{f(y\gv t)}{f(y\gv s)}\,\dd y
  = \mean[\ell(s,Y)] - \mean[\ell(t,Y)], \quad Y\sim f(y\gv t),
  $$
  the Kullback-Leibler distance from $f(y\gv s)$ to $f(y\gv t)$.  For
  $\cF$ with enough regularity, it is not hard to show $D$ has the
  following properties:
  \begin{enumerate}
  \item[1)] $D$, $\partial D/\partial s$, $\partial^2 D/\partial s^2$
    are  continuous in $(t,s)$;
  \item[2)] $D(t,t) =  (\partial D/\partial s)(t,t) = 0$,
    $I(t)=(\partial^2 D/\partial s^2)(t,t)>0$;
  \item[3)] every $t\in [a,b]$ is identifiable in $\cF$; and
  \item[4)] as $h\to 0$, $D(t,t+h)/h^2 \to I(t)$ uniformly for $t\in
    [a,b]$.
  \end{enumerate}
  
  Property 2) implies that for $s$ in  a neighborhood of $t$,
  $D(t,s)\ge I(t)(t-s)^2/2$.   Together with the other three
  properties and the compactness of $[a,b]\times [a,b]$, for some
  $C_\cF>0$, $D(t,s)\ge C_\cF(t-s)^2$ for all $t,s$.  Now for $i\le N$
  and $v\in D_0$, since $Y_i$ has density $f(y\gv X_i\tp\theta)$,
  $$
  \mean[\gamma_i(X_i\tp v, Y_i)] - \mean[\gamma_i(X_i\tp\theta, Y_i)] 
  = D(X_i\tp\theta, X_i\tp v)\ge C_\cF|X_i\tp\theta- X_i\tp v|^2.
  $$
  Then by the definition of $\risk(v)$, 
  $$
  \risk(v) - \risk(\theta) \ge C_\cF\sm i N |X_i\tp(v-\theta)|^2
  = C_\cF\normx{X(v-\theta)}_2^2,
  $$
  so \eqref{eq:diff-R-quadratic} is satisfied.

  Finally, if $\gamma_i$ defined in \eqref{eq:mle-loss} satisfies
  Assumptions \ref{a:regular} -- \ref{a:domain}, then by Corollary
  \ref{cor:m=1} and Theorem \ref{thm:lasso}, given $q_1, q_2\in (0,1)$
  with $q_1+q_2<1$, the following bound
  $$
  \normx{\est\theta-\theta}_2 \le
  \frac{(M_1+M_2)\sqrt{\normx{\theta}_0}}{N}
  \times \frac{2\sqrt{2+K^2} K d}{C_\cF \kappa^2 (K-1)} 
  $$
  holds with probability at least $1-q_1-q_2$, where $M_1$ and $M_2$
  are as follows.  Denote by $\eno V p$ the column vectors of $X$ and
  $$
  \Delta = \sup_{u,v\in D_0} \normx{u-v}_1.
  $$
  Denote 
  $$
  F_1 =  \mathop{\rm ess\sup}\Grp{
    \sup_t \mx i N\Abs{\dot\ell(t,Y_i)}
  }, \quad
  F_2 = \mathop{\rm ess\sup}\Grp{
    \mx i N \lipnorm{\dot\ell(\cdot, Y_i)}
  }.
  $$
  Note $d_1=\ldots = d_N=d$.  Then
  $$
  M_1 = A \sqrt{2\ln(2p)} + B\sqrt{2\ln(p/q_1)} + 8\phi \ln(p/q_1),
  \quad 
  M_2 = F_1 \sqrt{2N \ln(2p/q_2)},
  $$
  where
  $$
  A = 4F_2 \Delta \mx j p \normx{V_j}_2, \quad
  B = (F_2/2) \Delta\mx j p \normx{V_j}_2, \quad
  \phi = \min(2F_1, F_2 d\Delta/2).
  $$

  Up to a factor of $\sqrt{\ln(p/q_2)}$, $M_2 = O(\sqrt{N})$.
  Typically, for well designed $X$, $\mx j p\normx{V_j}_2 =
  O(\sqrt{N})$.  Therefore, $M_1 = O(\sqrt{N})$ up to a multiplicative
  factor $\sqrt{\ln(p/q_1)}$ and an additive remainder of order $\ln
  (p/q_1)$.  As a result, $\normx{\est\theta - \theta}_2$ is of order
  $\sqrt{\normx{\theta}_0/N}$ up to factors much smaller than
  $\sqrt{N}$ unless $p$ is extremely large. 

  Similar conclusions can be made if $f(y\gv X_i\tp \theta)$ is the
  density of $N(X_i\tp\theta, \sigma_0^2)$.  In this case, we can use
  Corollary \ref{cor:m=1g}.  For brevity, the detail is omitted.
  \qed
\end{example}

\section{Proofs} \label{sec:proof}
In this section we give proofs for the results in previous sections.
First, recall that for $q\in [1,\infty)$,
\begin{align} \label{eq:tensor-norm}
  \normx{a\otimes b}_q^q = \normx{a}_q^q \normx{b}_q^q,
\end{align}
and for $a_1, a_2\in\Reals^m$, $b_1, b_2\in\Reals^n$, $(a_1\tp
a_2)(b_1\tp b_2) = (a_1\otimes b_1)\tp (a_2 \otimes b_2)$, giving
\begin{align} \label{eq:tensor}
  (a_1\tp a_2)^k = (a_1^{\otimes k})\tp (a_2^{\otimes k}).
\end{align}

\subsection{Proofs for Section \ref{sec:lip}}
\begin{proof}[Proof of Theorem \ref{thm:lipschitz}]
  By \eqref{eq:tensor},  \eqref{eq:taylor-1} and \eqref{eq:taylor-2}
  are equivalent.  For notational brevity, we shall avoid explicit
  use of $d_j$.  For this reason, the domain $D_0$ is not the
  one to be directly worked on.  Rather, we shall consider
  \begin{align} \label{eq:D}
    D = \{(d_1 v_1, \ldots, d_p v_p)\tp: v\in D_0\}.
  \end{align}
  In other words, $D$ is the image of $D_0$ under the 1-1 transform
  $T: v\to (d_1 v_1, \ldots, d_p v_p)\tp$.  We shall use the
  $\ell_1$ norm on $D$.  Note that the norm induces a weighted
  $\ell_1$ norm on $D_0$ as
  $$
  \normx{u-v} = \normx{Tu-Tv}_1=\sm j p d_j |u_j-v_j|,
  $$
  which is the reason why $\sm j p d_j|u_j-\theta_j|$ appears in the
  expansions \eqref{eq:taylor-1} and \eqref{eq:taylor-2}.  Moreover,
  $R$ in \eqref{eq:R} can be expressed as the diameter of $D$ under
  $\ell_1$,
  $$
  R=\sup_{u,v\in D} \normx{u-v}_1.
  $$

  Based on the same consideration as \eqref{eq:D}, denote for $i\le
  N$, $j\le p$,
  \begin{align} \label{eq:XV}
    X_{ij} = h_j(Z_i)/d_j,
    \quad
    X_i = (X_{i1}, \ldots, X_{ip})\tp,
    \quad
    V_j = (X_{1j}, \ldots, X_{Nj})\tp.
  \end{align}
  Then Assumption \ref{a:bounded} on the boundedness of $h_j(Z_i)$
  implies
  \begin{align} \label{eq:normal}
    \prob\Cbr{|X_{ij}|\le 1, \forall i\le N,\, j\le p}=1.
  \end{align}
  Furthermore, for $v\in D$, let $u\in D_0$ such that $Tu=v$.  Then
  $X_i\tp v = h(Z_i)\tp u$, so we can easily translate an expansion in
  terms of $X_i\tp v$ into one in terms of $h(Z_i)\tp u$.  Therefore,
  until the end of the proof, we will focus on $D$.
  
  For brevity, for each $i\le N$, denote
  $$
  f_i(t) = \gamma_i(t, Y_i), \quad
  f_i\Sp k(t) = \frac{\partial^k \gamma_i(t, Y_i)}{\partial t^k},
  \quad k\le m+1.
  $$
  Fix $\theta\in D$ .  For $i\le N$ and $v$, define random vectors
  $c=(\eno c N)$ and $t=(\eno t N)$ with
  $$
  c_i = X_i\tp\theta, \quad t_i=X_i\tp(v-\theta).
  $$
  For $i\le N$, let $\varphi_i$ be the following random function on
  $\Reals$,
  \begin{align}
    \varphi_i(t) =
    \begin{cases}
      \displaystyle t^{-m} \Sbr{
        f_i(c_i+t) - \sm km \frac{f_i\Sp k(c_i)}{k!} t^k
      },
      &
      t\not=0; \\[1ex]
      0, & t=0.
    \end{cases}
    \label{eq:phi-taylor}
  \end{align}
  We need the following property of $\varphi_i$.
  \begin{lemma} \label{lemma:contraction}
    W.p.~1, each $\varphi_i\in C(a_i-c_i,b_i-c_i)$, and
    \begin{align} \label{eq:phi-bound}
      |\varphi_i(t)|\le 
      \min\Grp{\frac{2F_m}{m!}, \frac{F_{m+1}|t|}{(m+1)!}}
    \end{align}
    and $\lipnorm{\varphi_i}\le \psi$, where
    $$
    \psi= \psidef
    $$
  \end{lemma}
  
  Lemma \ref{lemma:contraction} will be proved later. 
  Clearly,
  \begin{align*}
    \sm i N\gamma_i(X_i\tp v, Y_i)
    &
    =\sm i N f_i(c_i+t_i)
    =\sm iN \Grp{
      \sm k m \frac{f_i\Sp k(c_i)}{k!} t_i^k
      + \varphi_i(t_i) t_i^m
    } \\
    &=
    \sm k m\nth{k!}\Grp{\sm i N f_i\Sp k(c_i) t_i^k}
    + \sm i N \varphi_i(t_i) t_i^m,
  \end{align*}
  where, by Assumption \ref{a:bounded}, w.p.~1, $t_i =
  X_i\tp(v-\theta)\in (a_i - c_i, b_i-c_i)$, $\forall i\le N$, $v\in
  D$.  Then by \eqref{eq:tensor}, 
  \begin{align*}
    \sm i N\gamma_i(X_i\tp v, Y_i) 
    =
    \sm km \nth{k!}
    \Grp{
      \sm i N f_i\Sp k(c_i) X_i^{\otimes k}
    }\tp (v-\theta)^{\otimes k}
    + \Grp{\sm i N \varphi_i(t_i) X_i^{\otimes m}}\tp
    (v-\theta)^{\otimes m}.
  \end{align*}
  Therefore,
  \begin{align} \label{eq:dev-taylor}
    \sm i N \dev{\gamma_i(X_i\tp v, Y_i)}
    = \sum_{k=1}^m \nth{k!} \sm i N
    \dev{f_i\Sp k(c_i) X_i^{\otimes k}}\tp (v-\theta)^{\otimes k}
    + \sm i N \dev{\varphi_i(t_i) X_i^{\otimes m}
    }\tp (v-\theta)^{\otimes m}.
  \end{align}
  
  By H\"older inequality and \eqref{eq:tensor-norm},
  \begin{align}
    \Abs{
      \sm i N \dev{\varphi_i(t_i) X_i^{\otimes m}}\tp
      (v-\theta)^{\otimes m}
    }
    \le 
    \Norm{
      \sm i N\dev{\varphi_i(t_i) X_i^{\otimes m}}
    }_\infty \Norm{v-\theta}_1^m.
    \label{eq:moment}
  \end{align}
  For each $\jmath = (\eno j p)$ with $j_s\le p$, denote
  $$
  X_{i\jmath} = X_{i j_1} \cdots X_{i j_m},
  $$
  where the product on the right hand side is defined to be 1 if
  $m=0$.  Then the coordinates of $X_i^{\otimes m}$ can be written as
  $X_{i\jmath}$, with $\jmath$ sorted, say, in the dictionary order.
  Let
  $$
  Z_\jmath = \sup_{v\in D} 
  \Abs{
    \sm i N \sdev{\varphi_i(t_i) X_{i\jmath}}
  }.
  $$
  Then from \eqref{eq:moment},
  \begin{align} \label{eq:moment-Z}
    \Abs{
      \sm i N\dev{\varphi_i(t_i) X_i^{\otimes m}}\tp
      (v-\theta)^{\otimes m}
    }
    \le
    \normx{v-\theta}_1^m
    \max_\jmath\Abs{\sm i N\sdev{\varphi_i(t_i) X_{i\jmath}}}
    \le \normx{v-\theta}_1^m \max_\jmath Z_\jmath.
  \end{align}
  
  By \eqref{eq:normal}, w.p.~1, $|X_{ij}|\le 1$, $i\le N$, $j\le p$, 
  and so $|t_i|=|X_i\tp(v-\theta)|\le \normx{v-\theta}_1 \le R$.  Then
  by Lemma \ref{lemma:contraction},
  $$
  |\varphi_i(t_i)| \le \min\Grp{\frac{2F_m}{m!},
    \frac{F_{m+1}R}{(m+1)!}}=\phi.
  $$
  It follows that
  \begin{align} \label{eq:dev-bound}
    |\varphi_i(t_i) X_{i\jmath}| \le \phi, \quad
    \Abs{\sdev{\varphi_i(t_i) X_{i\jmath}}}\le 2 \phi := M_0,
    \ \forall\,\jmath,  \quad\text{w.p.~1}.
  \end{align}
  
  Observe that given $v\in D$, for each $i\le N$, $\varphi_i(t_i)
  X_{i\jmath}$ is a function only in $(Y_i, Z_i)$.  Therefore, by
  independence, for $m\ge 0$ and $v\in D$,
  \begin{align*}
    \var\Grp{
      \sm i N \varphi_i(t_i) X_{i\jmath}
    }
    =\sm i N \var\Grp{
      \varphi_i(t_i) X_{i\jmath}
    }
    \le 
    \sm i N \mean\Sbr{
      \varphi_i(t_i)^2 X_{i\jmath}^2
    } 
    \le 
    \phi^2 \mean\Sbr{\sm i N  X_{i\jmath}^2}.
  \end{align*}
  If $m=0$, then the right hand side is $N\phi^2$.  If $m\ge 1$, by
  Young inequality,
  $$
  \sm i N X_{i\jmath}^2
  = \sm i N X_{ij_1}^2 \cdots X_{ij_m}^2
  \le \prod_{s\le m} \Grp{\sm i N X_{ij_s}^{2m}}^{1/m} 
  = \prod_{s\le m} \normx{V_{j_s}}_{2m}^2
  \le \mx j p \normx{V_j}_{2m}^{2m}.
  $$
  Therefore,
  \begin{align} \label{eq:dev-var}
    \var\Grp{
      \sm i N \varphi_i(t_i) X_{i\jmath}
    } \le S_0^2:=
    \begin{cases}
      \phi^2 N & m=0 \\
      \phi^2\mean\Sbr{\mx j p \normx{V_j}_{2m}^{2m}} & m\ge 1.
    \end{cases}
  \end{align}
  
  Fix one $\jmath = (\eno j p)$.  We next combine \eqref{eq:dev-bound}
  and \eqref{eq:dev-var} with measure concentration to bound the upper
  tail of $Z_\jmath$.  Again, note that given $v$, $\varphi_i(t_i)
  X_{i\jmath}$ is a function only in $(Y_i, Z_i)$, with
  $t_i=X_i\tp(v-\theta)$.  Let
  $$
  \cT = \{\tau=(\tau_{v,a}^1, \ldots, \tau_{v,a}^N): v\in D,
  \quad a\in \{-1, 1\}\},
  $$
  be a collection of functions parameterized by $D\times \{-1,1\}$
  mapping $(\cY\times\cZ)^N$ into $\Reals^N$, such that
  $$
  \tau_{v,a}^i(Y_i, Z_i) = a M_0^{-1} 
  \sdev{\varphi_i(t_i) X_{i\jmath}}, \quad i\le N.
  $$
  Then $Z_\jmath = M_0 \tilde Z$,  $S_0^2 = M_0^2 \tilde S^2$, with
  $$
  \tilde Z = \sup_{\tau\in \cT} \sm i N \tau^i (Y_i, Z_i),
  \quad
  \tilde S^2 = 
  \sup_{\tau\in \cT} \var\Grp{\sm i N \tau^i (Y_i, Z_i)}.
  $$
  
  From \eqref{eq:dev-bound}, for $v\in D$ and $a=\pm 1$, $\tau_{v,a}^i
  \in [-1,1]$.  Clearly, $\mean \tau_{v,a}^i(Y_i, Z_i)=0$.
  Furthermore, w.p.~1, $\tau_{v,a}^i(Y_i, Z_i)$ is continuous in $v$.
  Therefore, by dominated convergence argument, Theorem 1.1 in
  \cite{klein:05:ap} can be applied to $\tilde Z$.  Let $w = 2\mean
  \tilde Z + \tilde S^2 = 2 \mean Z_\jmath / M_0 + S_0^2 / M_0^2$.
  Then by \cite{klein:05:ap},
  \begin{align*}
    \prob\Cbr{Z_\jmath > \mean Z_\jmath + M_0 a}
    =
    \prob\Cbr{\tilde Z > \mean{\tilde Z} + a}
    \le \exp\Cbr{-\frac{a^2}{2 w+3a}},\quad\forall a>0.
  \end{align*}
  For $s>0$, $a=(1/2)(3 s + \sqrt{9s^2 + 8s w})$ is the unique
  positive solution to $a^2/(2 w+3a)=s$.  Using $\sqrt{a+b}\le
  \sqrt{a}+\sqrt{b}$ and $2\sqrt{ab} \le a+b$,
  \begin{align*}
    \mean Z_\jmath + M_0 a
    &
    \le \mean Z_\jmath + (M_0/2)(3s + \sqrt{9s^2} + \sqrt{8s w})
    \\
    &
    = \mean Z_\jmath + M_0 \Grp{
      3s + \sqrt{2s(2\mean Z_\jmath/M_0 + S_0^2/M_0^2)}
    }\\
    &
    \le \mean Z_\jmath + M_0 \Grp{
      3s + \sqrt{4s\mean Z_\jmath/M_0} + \sqrt{2s S_0^2/M_0^2}
    }\\
    &
    \le \mean Z_\jmath + M_0 (4s + \mean Z_\jmath/M_0  +
    (S_0/M_0)\sqrt{2s}).
  \end{align*}
  Then
  \begin{align} \label{eq:tail}
    \prob\Cbr{Z_\jmath > 2\mean Z_\jmath + S_0\sqrt{2s} + 4M_0
      s}\le e^{-s}.
  \end{align}
  
  To find an upper bound for $\mean Z_\jmath$, let $\eno\rx N$ be a
  Rademacher sequence independent of $(Y_i, Z_i)$.  By symmetrization
  inequality (cf. \cite{ledoux:91}, Lemma 6.3)
  \begin{align} \label{eq:symmetrize}
    \mean Z_\jmath \le 2 \mean \sup_{v\in D} \Abs{\sm i N
      \rx_i \varphi_i(t_i) X_{i\jmath}}.
  \end{align}
  By Fubini Theorem, the expectation on the right hand side is
  \begin{align*}
    \mean_{X,Y} \mean_\rx \sup_{v\in D} \Abs{\sm i N
      \rx_i \varphi_i(t_i) X_{i\jmath}},
  \end{align*}
  where $\mean_{X,Y}$ denotes the expectation only with respect to the
  (marginal) distribution of $(X_1, Y_1)$, \ldots, $(X_N, Y_N)$, and
  similarly for $\mean_\rx$.
  
  From \eqref{eq:phi-taylor}, $\varphi_i(0)=0$.  Assume $\psi>0$
  first.  Given $(X_1, Y_1)$, \ldots, $(X_N, Y_N)$, by Lemma
  \ref{lemma:contraction} and \eqref{eq:normal},
  $$
  t\to \varphi_i(t)X_{i\jmath}/\psi
  $$
  is a contraction for each $i\le N$.  Meanwhile, we can write
  $$
  \mean_\rx \sup_{v\in D}
  \Abs{\sm i N \rx_i \varphi_i(t_i) X_{i\jmath}} 
  = \mean_\rx \sup_{t\in T} \Abs{\sm i N \rx_i \varphi_i(t_i)
    X_{i\jmath}},
  $$
  with $T = T(\eno X N) = \{(\eno t N): t_i = X_i\tp (v-\theta),\ v\in
  D\}$.  Then by a comparison inequality (cf. Theorem 4.12 in
  \cite{ledoux:91}),
  \begin{align*}
    \mean_\rx
    \sup_{v\in D}\Abs{\sm i N \rx_i \varphi_i(t_i) X_{i \jmath}}
    \le 2\psi \mean_\rx \sup_{t\in T} \Abs{\sm i N \rx_i t_i}.
  \end{align*}
  Using $t_i = X_i\tp (v-\theta)$ and by the same argument for
  \eqref{eq:moment}
  \begin{align*}
    \mean_\rx \sup_{t\in T} \Abs{\sm i N \rx_i t_i}
    &= \mean_\rx \sup_{v\in D} \Abs{\sm i N \rx_i X_i\tp (v-\theta)}
    \nonumber \\
    &
    \le \mean_\rx \sup_{j\le p,\,v\in D}
    \Abs{\sm i N \rx_i X_{i j}}\normx{v-\theta}_1
    \le R \mean_\rx\max_{j\le p} \Abs{\rx\tp V_j},
  \end{align*}
  where $\rx = (\eno\rx N)\tp$.  With $(X_i, Y_i)$ being fixed, by a
  result in \cite{massart:00} (Lemma 5.2),
  $$
  \mean_\rx\max_{j\le p} \Abs{\rx\tp V_j}
  \le \sqrt{2\ln (2 p)} \mx j p \normx{V_j}_2.
  $$
  Combining the inequalities and taking expectation with respect to
  $(X_i, Y_i)$,
  \begin{align}
    \mean \sup_{v\in D}
    \Abs{\sm i N \rx_i \varphi_i(t_i) X_{i\jmath}}
    \le 2\sqrt{2}\psi R\sqrt{\ln(2p)}
    \mean\Sbr{\mx j p\normx{V_j}_2}.
    \label{eq:contraction}
  \end{align}
  
  If $\psi=0$, then $\varphi_i\equiv 0$ and the above inequality holds
  trivially.  Combining \eqref{eq:tail} -- \eqref{eq:contraction}
  yields
  \begin{align*}
    \prob\Cbr{Z_\jmath > M_1 \sqrt{2\ln(2p)} + \sqrt{2} S_0
      \sqrt{s} + 4 M_0 s}\le e^{-s}.
  \end{align*}
  where 
  \begin{align}\label{eq:M}
    M_1 = 8\psi R\mean\Sbr{\mx j p\normx{V_j}_2}.
  \end{align}
  
  Finally, since there are $p^m$ different values of $\jmath$,
  by union-sum inequality,
  \begin{align} \label{eq:tail2}
    \prob\Cbr{\max_\jmath Z_\jmath 
      > M_1\sqrt{\ln(2p)} + \sqrt{2} S_0 \sqrt{\ln(p^m/q)}
      + 4 M_0\ln(p^m/q)
    } \le q, \quad\forall\, q\in (0,1).
  \end{align}
  Note $M_1$, $S_0^2$ and $4M_0$ are exactly $A$, $B$ and $C$ in
  Theorem \ref{thm:lipschitz}.  Then by \eqref{eq:moment-Z}, the proof
  is complete.
\end{proof}

\begin{proof}[Proof of Corollary \ref{cor:m=1}]
  As in the proof of Theorem \ref{thm:lipschitz}, we consider the
  domain $D$ in \eqref{eq:D} and $X_i$ in \eqref{eq:XV}.  Still denote
  $c_i = X_i\tp\theta$.  For $m=1$, by Theorem \ref{thm:lipschitz},
  for $\theta$, $v\in D$
  $$
  \sm i N\dev{f_i(X_i\tp v)}
  = \sm i N\sdev{f_i(c_i)}
  + \sm i N\sdev{f_i'(c_i) X_i}\tp(v-\theta)
  + \xi(v)\normx{v-\theta}_1.
  $$
  By H\"older inequality,
  $$
  \Abs{
    \sm i N\sdev{f_i'(c_i) X_i}\tp(v-\theta)
  } \le 
  \normx{v-\theta}_1\mx j p \Abs{
    \sm i N \sdev{f_i'(c_i) X_{ij}}
  }.
  $$
  Given $j\le p$, $\sdev{f_i'(c_i) X_{ij}}$ are independent with mean
  0, and each $|f_i'(c_i)X_{ij}|\le F_1$.  Therefore, by Hoeffding
  inequality (\cite{pollard:84}, p.~191) and union-sum inequality,
  $$
  \prob\Cbr{
    \mx j p \Abs{
      \sm i N \sdev{f_i'(c_i) X_{ij}}
    } \ge t
  } \le 2 p \exp\Cbr{-\frac{t^2}{2N F_1^2}}.
  $$
  Given $q\in (0,1)$, let $t=\sqrt{N}F_1\sqrt{2\ln (2p/q)}$ to get
  the right hand side no greater than $q$.  Combining this with the
  bound for $\xi(v)$, the proof is complete.
\end{proof}

\begin{proof}[Proof of Theorem \ref{thm:lipschitz-g}]
  The proof is similar to that of Theorem \ref{thm:lipschitz}, so we
  will be brief.  Define domain $D$ as \eqref{eq:D} and $X_{ij}$,
  $X_i$, $V_j$ as in \eqref{eq:XV}.  Let  $c=(\eno c N)$ and $t=(\eno
  t N)$ with
  $$
  c_i = X_i\tp\theta, \quad t_i=X_i\tp(v-\theta).
  $$

  Define $\varphi_i$ as in \eqref{eq:phi-taylor}, however, note that
  the meaning of $f_i$ is different here.  In particular, $f_i$ are
  nonrandom and hence $\varphi_i$ are nonrandom as well.  In spite of
  this, Lemma \ref{lemma:contraction} still holds.  Corresponding to
  \eqref{eq:dev-taylor},
  \begin{align*}
    &
    \sm i N \gx_i f_i(h(Z_i)\tp v) \\
    &
    =
    \sm k m \nth{k!} \Grp{\sm i N\gx_i f_i\Sp k(c_i) X_i^{\otimes
        k}}\tp (v-\theta)^{\otimes k}
    + \Grp{\sm i N \gx_i\varphi_i(t_i) X_i^{\otimes m}}\tp
    (v-\theta)^{\otimes m}.
  \end{align*}
  The next step is to bound the upper tail probability of $\max_\jmath
  Z_\jmath$, where for $\jmath=(\eno j m)$,
  $$
  Z_\jmath = \sup_{v\in D} 
  \Abs{
    \sm i N \gx_i\varphi_i(t_i) X_{i\jmath}
  }, \quad \gx = (\eno\gx N)\tp.
  $$
  Write $\gx_i = \sigma_i \rx_i$, where $\sigma_i^2 = \var(\gx_i)\le
  \sigma_0^2$ and $\eno\rx N$ are i.i.d.~$\sim N(0,1)$.  Fix one
  $\jmath$.  Then
  $$
  Z_\jmath = Z(\rx) =
  \sup_{v\in D} 
  \Abs{
    \sm i N \rx_i \sigma_i\varphi_i(t_i) X_{i\jmath}
  }, \quad \rx = (\eno\rx N)\tp.
  $$
  The function $Z$ is Lipschitz on $\Reals^N$ under the Euclidean
  norm ($\ell_2$ norm), because for $a$, $b\in \Reals^N$,
  $$
  |Z(a) - Z(b)|\le
  \sup_{v\in D} 
  \Abs{
    \sm i N (a_i-b_i) \sigma_i \varphi_i(t_i) X_{i\jmath}
  }
  \le \normx{a-b}_2 \sigma_0 S_0,
  $$
  where, as in \eqref{eq:dev-var},
  $$
  S_0^2 =
  \begin{cases}
    \phi^2 N & m=0 \\
    \phi^2 \mx j p \normx{V_j}_{2m}^{2m} & m\ge 1.
  \end{cases}
  $$
  
  Now by a concentration inequality for Gaussian measure
  (\cite{ledoux:01}, p.~41)
  \begin{align} \label{eq:mc-gauss}
    \prob\Cbr{Z(\rx)\ge \mean Z(\rx) + r \sigma_0 S_0} \le
    \exp(-r^2/2), \quad\forall r>0.
  \end{align}
  
  By Lemma \ref{lemma:contraction} and $|X_{ij}|\le 1$,
  $t\to\varphi_i(t)X_{i\jmath}/\psi$ is a contraction with 0 being
  mapped to 0.  Then by a comparison result for Gaussian process 
  (\cite{ledoux:91}, Corollary 3.17 and (3.13))
  $$
  \mean Z(\rx) \le
  4\sigma_0 \psi \mean\sup_{v\in D} \Abs{
    \sm i N \rx_i t_i
  }
  = 4\sigma_0\psi\mean\sup_{v\in D} \Abs{
    \sm i N \rx_i X_i\tp(v-\theta)
  }
  \le 4 \sigma_0 R \psi\mean\mx j p \Abs{\rx\tp V_j},
  $$
  and 
  $$
  \mean\mx j p \Abs{\rx\tp V_j}
  \le 3 \sqrt{\ln p} \mx j p \sqrt{\var(\rx\tp V_j)}
  = 3 \sqrt{\ln p} \mx j p \normx{V_j}_2.
  $$
  Using an argument in \cite{massart:00}, one can get a bound
  for the expectation that is tighter for large $p$.
  \begin{lemma} \label{lemma:massart}
    There is
    $$
    \mean\mx j p \Abs{\rx\tp V_j}
    \le 2 \sqrt{\ln (2p)} \mx j p \normx{V_j}_2.
    $$
  \end{lemma}
  
  Now \eqref{eq:mc-gauss} can be written in terms of $Z_\jmath$.
  Then, as in \eqref{eq:tail2}, for $q\in (0,1)$,
  \begin{align} \label{eq:gauss-tail}
    \prob\Cbr{\max_\jmath Z_\jmath
      > \sigma_0 \Grp{
        M_1\sqrt{\ln(2p)} + \sqrt{2\ln(p^m/q)} S_0
      }
    } \le q,
  \end{align}
  where
  $$
  M_1 = 8 R \psi\mx j p \normx{V_j}_2.
  $$
  This then finishes the proof.
\end{proof}

\begin{proof}[Proof of Corollary \ref{cor:m=1g}]
  From Theorem \ref{thm:lipschitz-g}, it is seen that
  $$
  \sm i N \gx_i f_i(h(Z_i)\tp v)
  = \sm i N \gx_i f_i(h(Z_i)\tp\theta) + \zeta + \xi(v) \sm j p d_j
  |v_j - \theta_j|,
  $$
  where
  $$
  \zeta=\sm j p \Grp{\sm i N \gx_i f_i'(h(Z_i)\tp \theta) h_j(Z_i)}
  (v_j - \theta_j).
  $$
  Therefore, with $w_j$ being defined as in \eqref{eq:m=1gw},
  $$
  |\zeta| \le \sigma_0 F_1\sm j p w_j |v_j - \theta_j|
  \times \mx j p |W_j|,
  $$
  with
  $$
  W_j = \nth{\sigma_0 F_1 w_j}
  \sm i N \gx_i f_i'(h(Z_i)\tp \theta) h_j(Z_i).
  $$
  It is easy to see that each $W_j$ is Gaussian with mean 0 and
  variance no greater than 1.  As a result,
  $$
  \prob\Cbr{\mx j p |W_j| \ge t} \le p \exp(-t^2/2), \quad t\ge 0.
  $$
  Given $q\in (0,1)$, letting $t=\sqrt{2\ln(p/q)}$ then finishes the
  proof.
\end{proof}

\subsection{Proof for Section \ref{sec:app-lasso}}

\begin{proof}[Proof of Theorem \ref{thm:lasso}]
  For $A\subset\{1,\ldots,p\}$ and $v\in\Reals^p$, denote by $v_A$ the
  vector $u\in\Reals^p$ with $u_i = v_i \cf{i\in A}$.  By definition
  of $\est\theta$,
  $$
  \risk(\est\theta) - \risk(\theta)
  \le 
  \sm i N\dev{\gamma_i(X_i\tp\theta, Y_i)-
    \gamma_i(X_i\tp \est\theta, Y_i)
  }
  +\lambda d(\normx{\theta}_1 - \normx{\est\theta}_1).
  $$
  
  Let $\lambda = (1+1/c)M_qd$, where $c>0$ is to be determined.  Then,
  writing $r=1/c$, on the event that \eqref{eq:lasso-d} holds,
  \begin{align*}
    \risk(\est\theta) - \risk(\theta)
    \le M_q d \Sbr{\normx{\est\theta-\theta}_1 +
      (1+r) (\normx{\theta}_1 - \normx{\est\theta}_1)}.
  \end{align*}
  Fix any $J$ containing $\sppt(\theta)$.  Then
  \begin{align*}
    &
    \normx{\est\theta-\theta}_1 +
    (1+r) (\normx{\theta}_1 - \normx{\est\theta}_1) \\
    &
    = \sum_{i\in J} |\est\theta_i-\theta_i|
    + \sum_{i\not\in J} |\est\theta_i|+
    (1+r) \Grp{
      \sum_{i\in J} |\theta_i| - \sum_{i\in J} |\est\theta_i| -
      \sum_{i\not\in J} |\est\theta_i|
    }\\
    &
    = \sum_{i\in J} \Sbr{|\est\theta_i-\theta_i| +
      (1+r) (|\theta_i| - |\est\theta_i|)}
    -r \sum_{i\not\in J} |\est\theta_i| \\
    &
    \le (2+r) \normx{\est\theta_J-\theta}_1-
    r\normx{\est\theta_{J^c}}_1.
  \end{align*}
  
  On the one hand, the above inequalities yield
  \begin{align} \label{eq:lasso}
    \risk(\est\theta) - \risk(\theta)
    \le M_q d (2+1/c) \normx{\est\theta_J - \theta}_1,
  \end{align}
  and on the other, since by definition of $\theta$,
  $\risk(\est\theta)\ge \risk(\theta)$,
  \begin{align} \label{eq:dantzig}
    \normx{\est\theta_{J^c}}_1 \le (1+2c)
    \normx{\est\theta_J-\theta}_1.
  \end{align}
  
  Set $c = (K-1)/2$.  Then $\lambda=(1+1/c)M_q d$ is as in
  \eqref{eq:lasso-par}.  By \eqref{eq:RE}, \eqref{eq:diff-R-quadratic}
  and \eqref{eq:lasso}, for any $J\supset\sppt(\theta)$ with $|J|\le 2
  \normx{\theta}_0$,
  $$
  N C_\gamma \kappa^2 \normx{\est\theta_J - \theta}_2^2
  \le \risk(\est\theta) - \risk(\theta)
  \le \frac{2M_q Kd}{K-1} \normx{\est\theta_J - \theta}_1.
  $$
  Since $\normx{\est\theta_J - \theta}_1 \le \sqrt{|J|}
  \normx{\est\theta_J - \theta}_2$, it follows that
  \begin{align} \label{eq:L2}
    \normx{\est\theta_J - \theta}_2 \le b\sqrt{|J|} \quad
    \text{with}\quad b=\frac{M_q}{N} \times
    \frac{2K d}{C_\gamma \kappa^2(K-1)}.
  \end{align}
  
  Let $A$ be the set of indices $i\not\in\sppt(\theta)$ corresponding
  to the $\normx{\theta}_0$ largest $|\est\theta_i|$.  Then
  \eqref{eq:L2} holds for both $J_0=\sppt(\theta)$ and
  $J_1=\sppt(\theta)\cup A$.  It is well known that
  (cf.\ \cite{candes:tao:07})
  $$
  \normx{\est\theta_{J_1^c}}_2^2 \le
  \frac{\normx{\est\theta_{J_0^c}}_1^2}{\normx{\theta}_0}.
  $$
  By \eqref{eq:dantzig} and Cauchy-Schwartz inequality followed by
  \eqref{eq:L2}, 
  \begin{align*} 
    \normx{\est\theta_{J_1^c}}_2^2 
    \le
    \frac{K^2\normx{\est\theta_{J_0}-\theta}_1^2}{\normx{\theta}_0}
    \le K^2 \normx{\est\theta_{J_0}-\theta}_2^2
    \le K^2 b^2 \normx{\theta}_0.
  \end{align*}
  Combining this with \eqref{eq:L2} applied to $J=J_1$,
  \begin{align*}
    \normx{\est\theta - \theta}_2^2 
    =\normx{\est\theta_{J_1}-\theta}_2^2
    + \normx{\est\theta_{J_1^c}}_2^2
    \le b^2 |J_1| + K^2 b^2\normx{\theta}_0 =
    (2+K^2)b^2 \normx{\theta}_0.
  \end{align*}
  So we finally arrive at \eqref{eq:L2-total}.
\end{proof}

\subsection{Proof of Lemmas}
\begin{proof}[Proof of Lemma \ref{lemma:contraction}]
  If $m=0$, then $\varphi_i(t) = f_i(c_i+t) - f_i(c_i)$.  From 
  Assumptions \ref{a:regular} and \ref{a:domain}, the result is
  straightforward.

  Let $m\ge 1$.  For $t>0$, by Taylor expansion with an integral
  remainder, 
  \begin{align} \label{eq:taylor-integral}
    f_i(c_i+t) - \sm k m \frac{f_i\Sp k(c_i)}{k!} t^k
    = \nth{(m-1)!} \int_0^t (t-s)^{m-1} [f_i\Sp m(c_i+s) - f_i\Sp
    m(c_i)] \,\dd s,
  \end{align}
  yielding 
  $$
  \varphi_i(t)= \frac{t^{-m}}{(m-1)!} \int_0^t (t-s)^{m-1} [f_i\Sp
  m(c_i+s) - f_i\Sp m(c_i)] \,\dd s.
  $$
  Therefore, by Assumption \ref{a:regular}, on the one hand,
  \begin{align*}
    |\varphi_i(t)|
    \le 
    \frac{t^{-m}}{(m-1)!} \int_0^t (2F_m) (t-s)^{m-1} \,\dd s
    = \frac{2F_m}{m!},
  \end{align*}
  and on the other, 
  \begin{align*}
    |\varphi_i(t)|
    \le 
    \frac{t^{-m}}{(m-1)!}
    \int_0^t (t-s)^{m-1} (F_{m+1} s) \,\dd s
    = \frac{F_{m+1} |t|}{(m+1)!}.
  \end{align*}
  The inequalities hold likewise for $t<0$.  Therefore,
  \eqref{eq:phi-bound} holds.  The above inequality also implies that
  $\varphi_i$ is continuous at 0.  It is clear that $\varphi_i(t)$
  is continuous at $t\not=0$.  Thus $\varphi_i\in C(a_i-c_i,
  b_i-c_i)$.  

  It remains to show $\lipnorm{\varphi_i}\le \psi$.  Since $\varphi_i$
  is differentiable at $t\not=0$, it is enough to show
  $|\varphi_i'(t)|\le \psi$ for $t\not=0$.  First, let $m=1$.  For
  $t\not=0$,
  $$
  \varphi_i'(t) = t^{-2}[f_i(c_i) - f_i(c_i+t) + t f_i'(c_i+t)]
  = t^{-2} \int_0^t [f_i'(c_i+t) - f_i'(c_i+t-s)]\,\dd s.
  $$
  By Assumption \ref{a:regular}, $|f_i'(c_i+t)-f_i'(c_i+t-s)|\le F_2
  |s|$.  Consequently $|\varphi_i'(t)|\le F_2/2=\psi$.

  Finally, let $m\ge 2$.  Define $g(t)=mf_i(c_i+t) - t f_i'(c_i+t)$.
  Then for $k<m$,
  $$
  g\Sp k(t) = (m-k) f_i\Sp k(c_i+t) - t f_i\Sp{k+1}(c_i+t)
  $$
  and then
  \begin{align*}
    \varphi_i'(t) &
    =
    t^{-m} \Grp{f_i'(c_i+t) - \sum_{k=1}^m
      \frac{f_i\Sp k(c_i)t^{k-1}}{(k-1)!}
    }
    - m t^{-m-1}\Grp{
      f_i(c_i+t) - \sum_{k=0}^m \frac{f_i\Sp k(c_i) t^k}{k!}
    } \\
    &
    = - t^{-m-1} 
    \Grp{
      m f_i(c_i+t) - t f_i'(c_i+t) 
      - \sum_{k=0}^{m-1} \frac{(m-k) f_i\Sp k(c_i) t^k}{k!}
    } \\
    &
    = - t^{-m-1} \Grp{ 
      g(t) - \sum_{k=0}^{m-1} \frac{g\Sp k(0) t^k}{k!}
    } \\
    &
    = -\frac{t^{-m-1}}{(m-2)!} \int_0^t (t-s)^{m-2}
    [g\Sp{m-1}(s) - g\Sp{m-1}(0)]\,\dd s,
  \end{align*}
  where the last equality is by similar Taylor expansion as
  \eqref{eq:taylor-integral}, now applied to $g$ with order $m-1$.
  For each $s$,
  \begin{align*}
    g\Sp{m-1}(s) - g\Sp{m-1}(0) 
    &= f_i\Sp{m-1}(c_i+s) - s f_i\Sp m(c_i + s) - f_i\Sp{m-1}(c_i) \\
    &
    = \int_0^s [f_i\Sp m(c_i+s-u) - f_i\Sp m (c_i+s)]\,\dd u,
  \end{align*}
  giving $|g\Sp{m-1}(s) - g\Sp{m-1}(0)|\le F_{m+1} s^2/2$.  Then
  $$
  |\varphi_i'(t)| \le \frac{t^{-m-1} F_{m+1}}{2(m-2)!} \int_0^t
  (t-s)^{m-2} s^2\,\dd s = \frac{F_{m+1}}{m!} = \psi.
  $$
  This finishes the proof.
\end{proof}

\begin{proof}[Proof of Lemma \ref{lemma:massart}]
  Let $x = \mean \mx j p |\rx\tp V_j|$.  By Jensen inequality, for any
  $t>0$,
  $$
  \exp(t x) \le \mean\Sbr{\exp\Grp{t \mx j p |\rx \tp V_j|}}
  =\mean \Sbr{\mx j p \exp(t |\rx \tp V_j|)}
  \le \sm j p \mean[\exp(t|\rx \tp V_j|)].
  $$
  Since $\rx \tp V_j \sim N(0, \normx{V_j}_2^2)$,
  $$
  \mean[\exp(t|\rx \tp V_j|)]
  \le
  \mean[\exp(t\rx \tp V_j)]
  +\mean[\exp(-t\rx \tp V_j)] = 2 \exp(t\normx{V_j}_2^2).
  $$
  Then
  $$
  \exp(tx) \le 2 p \exp\Grp{t^2\mx j p \normx{V_j}_2^2}.
  $$
  The proof is finished by letting $t = x/(2\mx j p
  \normx{V_j}_2^2)$.
\end{proof}

\bibliographystyle{plain}

\end{document}